\documentclass[a4paper,10pt]{amsart} 
\usepackage{amsmath} 
\usepackage{amssymb} 
\usepackage{amsthm} 
\usepackage{comment}

\usepackage{xcolor}
\usepackage{hyperref}
\hypersetup{
	colorlinks,
	linkcolor={red!50!black},
	citecolor={blue!50!black},
	urlcolor={blue!80!black},
	breaklinks=true
}

\newtheorem{theorem}{\bf Theorem}[section]
\newtheorem{definition}[theorem]{\bf Definition}
\newtheorem{lemma}[theorem]{\bf Lemma}
\newtheorem{corollary}[theorem]{\bf Corollary}

\newtheorem{conjecture}[theorem]{\bf Conjecture}

\newcommand{\N}{\mathbb{N}}

\newcommand{\p}{\mathbb{P}}

\newcommand{\F}{\mathbb{F}}

\newcommand{\con}{\mathrm{con}}
\newcommand{\dom}{\mathrm{dom}}

\title{An application of Bertini Theorem}

\author[M. Makhul]{
	Mehdi Makhul $^{\ast, \circ}$}
\address[Mehdi Makhul]{Johann Radon Institute for Computational and Applied 
	Mathematics (RICAM), Austrian Academy of Sciences}

\author[J. Schicho]{
	Josef Schicho$^{\ast}$}
\address[Josef Schicho]{Research Institute for Symbolic Computation (RISC), 
	Johannes Kepler University, Linz}
\thanks{Mehdi Makhul $^{\ast, \circ}$ was supported by the Austrian ScienceFund  FWF Project P 30405-N32 and Josef Schich $^\ast$ Supported by the Austrian Science Fund (FWF): W1214-N15, 
	Project DK9} 
\email{\{mmakhul,jschicho\}@risc.jku.at}

\begin{document}
\begin{abstract}
Given an irreducible variety $X$ over a finite field, the density of hypersurfaces of varying degree $d$ intersecting $X$ in an irreducible subvariety is $1$, by a result of Charles and Poonen. In this note, we analyse the situation fixing $d=1$ and extend the base field instead of the degree $d$. We compute the probability that a random linear subspace of the right dimension intersects $X$ in a given number of points. 
\end{abstract}
	
\maketitle
\section{Introduction}

Classical Bertini theorems over an infinite field $K$ assert that if a subscheme $X \subset \p^n(K)$ has a certain property (smooth, geometrically irreducible), then for almost all hyperplanes $\Gamma$, the intersection $X \cap \Gamma$ has this property too. 

In this paper, we consider an algebraic variety~$X\subset \p^n$ of degree~$d$ and dimension $m$ over a finite field~$\F_q$ with $q$ 
elements, where $q$ is a prime power. Given such a variety, we compute the probability for which a codimension $m$ linear subspace  in~$\p^n$ intersects the variety in exactly $k$ points. Notice that here we consider the mere set-theoretic intersection: no multiplicities are taken into account. We can then consider the same kind of probability, keeping the same variety $X$, but changing the base field from~$\F_q$ to~$\F_{q^2}$, $\F_{q^3}$ and so on. In this way, for every $N \in \N$ we define the number~$p_k^N(X)$, namely the probability for a codimension $m$ linear subspace  in~$\p^n$ defined over $\F_{q^N}$ to intersect~$X$ in exactly~$k$ points. If the limit as~$N$ goes to infinity of the sequence $\bigl( p_k^N(X) \bigr)_{N \in \N}$ exists, we 
denote this number by~$p_k(X)$. We will compute the exact values $p_k(X)$ for each $k=0,1,\dots d$, provided that $X$ satisfies some geometrical properties.


The main result of this paper is the following theorem
%
%
%
%

\begin{theorem}
	\label{main}
	Let $X$ be a geometrically irreducible variety of dimension $m$ and degree $d$ in the projective space $\p^n$. Suppose that $X$ has simple tangency property, then for every $k\in \{0,\dots d\}$ we have
	\[
	p_k(X) = \sum_{s=k}^d \frac{(-1)^{k+s}}{s!} \binom{s}{k}.
	\]
\end{theorem}
This is a generalization of \cite[Theorem $1.2$]{Makhul2018}, which is the statement for the case when $X$ is a planar curve. The notion of simple tangency property will be define in Section \ref{subsec:simple-tangency}.

It has been shown that if $K$ is a finite field, then the Bertini Theorem about irreducibility can fail, see~\cite[Theorem $1.10$]{Charles2016}. In \cite{Charles2016}, the authors considered the density of hypersurfaces (of sufficiently high degree) whose intersection with a given geometrically irreducible variety is also geometrically irreducible. More precisely: Let $S_d\subset\F_q[x_0,\dots,x_n]$ be the set of homogeneous polynomials of degree $d$. For $f\in S_d$, let $Z(f)$ be the set of vanishing points in $\p^n$ of $f$. Then

\begin{theorem}[Charles--Poonen]
	\label{poonen1}
For a geometrically irreducible variety\\ $X\subset\p^n$ of dimension at least $2$ we have
\[
\lim_{d\to\infty}\frac{|\{f\in S_d : Z(f)\cap X \text{ is geometrically irreducible}\}|}{|S_d|}=1.
\] 

\end{theorem}  
Recently, an analogue of this problem for plane curves was investigated in \cite{Asgarli2019}.



Throughout this paper, when we write $J_m=G(n-m,n)$ we mean the variety of all linear subspaces of codimension $m$ in the projective space $\p^n$, the so-called~\emph{Grassmannian}. A \emph{Chow variety} is a variety whose points correspond to all cycles of a given projective space of given dimension and degree.


\section{Proof of the Main result}
\label{sec:main-result}

Throughout, unless otherwise stated, we let $X \subset \p^n$ be a geometrically irreducible algebraic variety defined over $\F_q$. Let $S^N$ be the set of hyperplanes in $\p^n$. Define
\begin{equation}
\label{eq:bertini}
\mu(X):=\lim_{N \to \infty}\frac{\Big|\big\{\Gamma \in S^N : \Gamma \cap X  \quad \text{is geometrically irreducible}\big\}\Big|}{\Big|\check{\p}^{n}(\F_{q^N})\Big|}
\end{equation}

where $\check{\p}^{n}$ is the dual of $\p^n$ and for every variety $Y$ defined over $F$ and field extension $E$ we write $X(E)$ for the set of all $E$-rational points of $Y$,  i.e., the set of all points in $Y$ with coordinates in $E$.

By applying Bertini's Theorem for an infinite field~\cite[Theorem~ $6.3(4)$]{Jouanolou1983} we show that $\mu(X)=1$. In other words, as $N$ approaches infinity, the intersection $X \cap \Gamma$ is geometrically irreducible, for a generic hyperplane $\Gamma$.

\begin{lemma}\label{lm:density}
Let $X$ be as above, then $\mu(X)=1$.

\end{lemma}
\begin{proof}
Define $m$ to be the dimension of $X$ and $d$ the degree of $X$. Let $\mathcal{H}_{d,m-1}$ be the Chow variety of cycles in $\p^n$ of dimension $m-1$ and degree $d$. Let $\Omega$ be the set of hyperplanes in $\p^n(\F_{q^N})$ whose intersection with $X$ is reducible. We know that if $\Gamma \in \Omega$ then $\dim(\Gamma \cap X)=m-1$. More precisely if $\Gamma \in \Omega$, then $\Gamma \cap X=X_1 \cup X_2$, where $X_i \in \mathcal{H}_{d_i,m_i}$  for $i=1,2$ and $\max(m_1,m_2)=m-1$ and $d_1+d_2=d$.

First we show that $\Omega$ is a closed set. To do this, consider the rational map  
\[
\Phi : S^N \dashrightarrow \mathcal{H}_{d,m-1} \quad \Gamma \mapsto \Gamma \cap X.
\]
Indeed 
\[\Omega= \bigcup_{d_1+d_2=d}\Phi^{-1}(\mathcal{H}_{d_1,m_1}\times \mathcal{H}_{d_2,m_2}).\]

Hence, $\Omega$ is an algebraic variety.

By Bertini's Theorem about irreducibility $\dim \Omega < \dim S^N=n+1$.  Hence, the probability that an element in $S^N$ is in $\Omega$ tends to $0$. More precisely, by the Lang-Weil Theorem this probability is bounded by
\[
\frac{(q^N)^{\dim \Omega}}{(q^N)^{\dim S^N}} \to 0 \quad \text{for $N \to \infty$}. \qedhere
\] 
\end{proof}

\subsection{Simple tangency and reflexivity}
\label{subsec:simple-tangency}
In this section we recall the basic properties of reflexive curves and its relation with curves with simple tangency that we will need.

Let $X$ be a variety in the projective space $\p^n$,defined over an  algebraically  closed perfect field $K$. We define the \emph{conormal variety} of $X$ as the Zariski closure of the set 
\[
\con(X) := \overline{\big\{(p,\Gamma) \in X \times \check{\p}^n: T_p(X) \subset \Gamma \big\}}.
\] 

Let $\pi_2$ be the second projection which is called the \emph{conormal map} and define\\ $X^*:=\pi_2(\con(X))$. If $\con(X)$ and $\con(X^*)$ are isomorphic by the map which flips the two entries of a pair in a product variety, then we say that $X$ is \emph{reflexive}.

It is known that if the field $K$ has zero characteristic, then every variety is reflexive; this is not true in characteristic $p>0$. The following theorem is useful for checking if a given projective variety is reflexive or not. see \cite{Wallace1956}.

\begin{theorem}[Monge-Segre-Wallace]
	\label{thm:MSW}
A projective variety $X$ is reflexive if and only if the conormal
map $\pi_2$ is separable.
\end{theorem}
In \cite{Hefez1985a} the authors proved the following result, called the Generic Order of Contact Theorem:
\begin{theorem}
A projective curve $Z$ is non-reflexive if and only if for a general point $p$ of $Z$ and a general tangent hyperplane $H$ to $Z$ at $p$, we have 
\[
[K(\con(Z)):K(Z^*)]_{isep}=I(p,Z.H).
\]
Where $I(p,Z.H)$ is the intersection multiplicity of $Z$ and $H$ at $p$, and~$[K(\con(Z)):K(Z^*)]_{isep}$ is the inseparable degree extension.
\end{theorem}
A combination of these two theorems implies 

\begin{corollary}\label{cor:simple-tangency}
If $C$ is a geometrically irreducible reflexive curve of degree $d$ in~$\p^n$, then there exists a hyperplane $H \subset \p^n$ intersecting $C$ in $d-1$ smooth points of $C$ such that  $H$ intersects $C$ transversely at $d-2$ points and has intersection multiplicity $2$ at the remaining point.
\end{corollary}

\begin{proof}
Since $C$ has simple tangency, $\pi_2$ is separable, hence for all $p \in C$ except finitely many of them the tangent line at $p$ has multiplicity at most~$2$~\cite[Proposition $1.5\, (a)$]{Pardini1986} at $p$. On the other hand, separability of $\pi_2$ implies there are at most finitely many lines tangent to $C$ at more than one points~\cite[Proposition $1.5 \, (b)$]{Pardini1986}. This completes the proof. 
\end{proof}

\newpage
\begin{definition}~
	\label{def:simple-tangency}
\begin{itemize}
	\item A curve $C$ having the property discussed in Corollary \ref{cor:simple-tangency} is said to have \emph{simple tangency}.
		
	\item Let $X$ be a geometrically irreducible variety in $\p^n$ of dimension~$m$. We say that $X$ has the \emph{simple tangency property} if there exist a linear subspace $\Gamma \in J_{m-1}$ such that the curve $ X \cap \Gamma$ has simple tangency.
\end{itemize}

\end{definition}

\begin{lemma}\label{lm:surface-simple-tangency}
Suppose that $X$ is a geometrically irreducible variety  of degree $d$ and dimension $m$ in $\p^n$ with simple tangency property. Then for a general linear subspace $\Gamma \in J_{m-1}$ the intersection $X \cap \Gamma$ is a curve with simple tangency.
\end{lemma}
\begin{proof}
Let $\mathcal{H}_{d,m}$ be the Chow variety. Let $\mathcal{H}_{d,1}^{\prime}$ be the set of all curves in $\p^n(K)$ of degree~$d$ and without simple tangency. Define the rational map
\[
\Phi_X \colon J_{m-1} \dashrightarrow \mathcal{H}_{d,1}, \qquad \Gamma \mapsto X \cap \Gamma.
\]
Notice that $\Phi_X$ in general is not a morphism but we can consider the restriction of $\Phi_X$ to the set $\dom(\Phi_X)$ where $\Phi_X$ is defined to get a morphism. For a fixed $X \in \mathcal{H}_{d,m}$, define
\[
\Omega_X:=\Big\{\Gamma \in J_{m-1}: \quad X \cap \Gamma \quad \text{does not have simple tangency} \Big\}.
\]

By the definition we have  $\Omega_X \subset \Phi^{-1}(\mathcal{H}_{d,1}^{\prime})$. Since $X$ is a variety with simple tangency property there exists a linear subspace~$\Gamma \in J_{m-1}$ such that $\Gamma \cap X$ is a curve with simple tangency, hence $\Phi^{-1}(\mathcal{H}_{d,1}^{\prime})$ is a proper set. We need to only show that $\Phi^{-1}(\mathcal{H}_{d,1}^{\prime})$ is a closed set. The proof of the lemma is a consequence of the following claim.

\textbf{Claim}. $\mathcal{H}_{d,1}^{\prime}$ is a closed set in $\mathcal{H}_{d,1}$.

\textit{Proof of the claim.}
If a curve $C$ is in $\mathcal{H}_{d,1}^{\prime}$, then its  conormal map is not separable. This is the case if and only if the Jacobian of the conormal map vanishes identically. This can be expressed as algebraic equations in the curve, hence the set $\mathcal{H}_{d,1}^{\prime}$ is closed.
\end{proof}
Let us first formulate the definition of the probabilities that we want to compute. 
\begin{definition}[Probabilities of intersection]
	\label{def:density-of-variety}
	Let $q$ be a prime power and let $X$ be a geometrically irreducible variety of dimension $m$ and degree~$d$ defined over~$\F_q$. For every $N \in \N$ and for every $k \in \{0, \dotsc, d\}$, the \emph{$k$-th probability of intersection} $p_k^N(X)$ of varieties of codimension $m$ with~$X$ over~$\F_{q^N}$ is  
	\[
	p_k^N(X) := 
	\frac{\Bigl| 
		\bigl\{
		V \in J_m  \, : \, 
		|X(\F_{q^N}) \cap V(\F_{q^N}) | = k 
		\bigr\} 
		\Bigr|}{|J_m(\F_{q^N})|} \,.
	\]
Moreover, we define $p_k(X):= \lim_{N \to \infty} p_k^N(X)$.
\end{definition}

%
%

\textbf{Proof of Theorem \ref{main}}.
 Define 
\[
I:=\big\{(V,W)\in J_m \times J_{m-1} : V \subset W\big\}.
\]
From Definition \ref{def:density-of-variety}, we have
\[
p_k^N(X)=\frac{\big|\big\{V \in J_m: |X(\F_{q^N}) \cap V(\F_{q^N})|=k\big\}\big|}{|J_m(\F_{q^N})|}=\frac{\big| \big\{(V,W) \in I: |X(\F_{q^N})\cap V(\F_{q^N})|=k\big\}\big|}{|I|}.
\]

We can write 
\begin{equation}\label{eq:density-1}
p_k^N(X)=\sum_{W\in J_{m-1}} \frac{\big| \big\{(V,W)\in I : |V\cap X\cap W|=k\big\}\big|}{|I|}.
\end{equation}
Where for a generic $W$ the intersection $X \cap W$ is a geometrically irreducible curve by Lemma \ref{lm:density}.
 Since any two linear subspaces have the same number of points over a finite field we can write:

\begin{equation}\label{eq:density-2}
p_k^N(X)=\sum_{W\in J_{m-1}} \frac{\big|\big\{V \subset W: |V\cap X \cap W|=k \big\} \big|}{\big|\big\{V: V \subset W \big\}\big|} \Bigg/ \frac{|I|}{\big|\big\{V: V\subset W_0 \big\}\big|}
\end{equation}

where $W_0$ is a fixed element in $J_{m-1}$. However, the denominator of Equation~(\ref{eq:density-2}) is $|J_{m-1}|$. Let us now write $J_{m-1}=A \sqcup B$, where
\[
A=\big\{W \in J_{m-1}: X \cap W \quad \text{is irreducible and has simple tangency}\big\},
\]
and 
\[
B=\big\{W \in J_{m-1}: X \cap W \quad \text{is reducible or without simple tangency} \big\}.
\]
From these and Equation~(\ref{eq:density-2}) we obtain
\[
p_k^N(X)=\frac{\sum_{W\in A} p_k^N(X\cap W)+\sum_{W\in B}\delta}{|J_{m-1}|},
\]
where $\delta$ is a number in interval $[0,1]$ that may depend on $W$. Note that when we write $p_k^N(X \cap W)$, we consider $X \cap W$ in the ambient space $W$ (not $\p^n$). Hence 
\[
p_k(X)= \lim_{N \to \infty}\frac{p_k^N(X\cap W_0)|A|+\delta|B|}{|J_{m-1}|}
\]
by the result in  \cite[Proposition $5.2$]{Makhul2018}. By Lemma \ref{lm:surface-simple-tangency} we know that $\frac{|B|}{|J_{m-1}|}\to 0$, when $N \to \infty$. This implies 
\[
p_k(X\cap W_0)=\lim_{N \to \infty}\frac{p_k^N(X \cap W_0)|A|}{|J_{m-1}|}
\]
and again by \cite[Proposition $5.2$]{Makhul2018} we obtain the result. $\square$

\vspace{5mm}
It is natural to consider the probabilities of intersection of a variety $X$ of degree $d$ and dimension $m$ in $\p^n$ with a random variety~$Y$ of 
degree~$e$ and codimension $m$ in~$\p^n$. If $X$ is a hypersurface and $Y$ is a curve, via the \emph{Veronese map} we can reduce this situation to the one of \cite[Proposition $5.2$]{Makhul2018}. This motivates us to pose the following conjecture.

\begin{conjecture}
Let $X$ be a geometrically irreducible variety of dimension $m$ and degree $d$ in $\p^n(\F_q)$, where $q$ is a prime power. Suppose that $X$ has the simple tangency property. Let $e \in \N$ be a natural number. Then for every $k \in \{0,1,\dots ek\}$ the probability that a random irreducible variety of degree $e$ and codimension $m$ intersects $X$ in exactly $k$ points is given by

\[
p_k(X,e) = \sum_{s=k}^{de} \frac{(-1)^{k+s}}{s!} \binom{s}{k}.
\]

\end{conjecture}

If we removed the word "irreducible" in the conclusion above, then the conjecture would be false:
the set of varieties in $\p^n$ of degree $d$ and codimension $m$ is bijective to the set of points in a Zariski-dense and open set of a Chow variety. The Chow variety has several components of maximal dimension; the smallest case for which this happens is $m=e=2$ and $n=3$. Here there is an 8-dimensional set of irreducible conics and an 8-dimensional set of reducible conics (pairs of lines). One can show that the probability that an irreducible conic intersect $X$ in $k$ points is as stated in the conjecture, but for the reducible conics, the probabilities differ. The total probabilities would be the arithmetic means of both, which would then also differ from the statement above.

\bigskip
\noindent {\Large \textbf{Acknowledgements}}. We are grateful to Herwig Hauser, who informed us about theorem of Bertini. Also we would like to thank Jose Capco for reading first draft and improving it and Matteo Gallet for helpful discussions.  

\providecommand{\bysame}{\leavevmode\hbox to3em{\hrulefill}\thinspace}
\providecommand{\MR}{\relax\ifhmode\unskip\space\fi MR }
\providecommand{\MRhref}[2]{%
	\href{http://www.ams.org/mathscinet-getitem?mr=#1}{#2}
}
\providecommand{\href}[2]{#2}

\end{document}